\documentclass[12pt]{article}

\makeatletter
\newfont{\footsc}{cmcsc10 at 8truept}
\newfont{\footbf}{cmbx10 at 8truept}
\newfont{\footrm}{cmr10 at 10truept}
\makeatother
\usepackage{amsmath}
\usepackage{amssymb}
\usepackage{exscale}
\usepackage{amsthm}
\usepackage{epsfig}
\usepackage{verbatim}
\usepackage{setspace}
\usepackage{bbm}
\usepackage{bm}
\usepackage{framed}
\usepackage{xcolor}

\theoremstyle{plain}
\newtheorem{theorem}{Theorem}[section]
\newtheorem{proposition}[theorem]{Proposition}
\newtheorem{lemma}[theorem]{Lemma}
\newtheorem{example}[theorem]{Example}
\newtheorem{corollary}[theorem]{Corollary}

\theoremstyle{definition}

\newtheorem{remark}[theorem]{Remark}

\DeclareMathOperator{\spn}{span}

\def\bs0{\bf 0}

\title{ A Short Note on Compact Embeddings of Reproducing Kernel Hilbert Spaces in $L^2$ and Infinite-variate Function Approximation.}

\author{Marcin Wnuk \footnote{Universit\"at Osnabr\"uck, Institut f\"ur Mathematik, Albrechtstr. 28a, D-49076, Osnabr\"uck,  
e-mail: marcin.wnuk@uni-osnabrueck.de }}

\begin{document}

\maketitle
\begin{abstract}
This note consists of two largely independent parts. 
In the first part we give conditions on the kernel $k: \Omega \times \Omega \rightarrow \mathbb{R}$ of a reproducing kernel Hilbert space $H$
continuously embedded via the identity mapping into $L^2(\Omega, \mu),$
which are equivalent to the fact that $H$ is even compactly embedded into $L^2(\Omega, \mu).$ 

In the second part we consider a scenario from infinite-variate $L^2$-approximation.
 Suppose that the embedding of a reproducing kernel Hilbert space of univariate functions with reproducing kernel $1+k$ into $L^2(\Omega, \mu)$ is compact. We provide a simple criterion for checking compactness of the embedding of a reproducing kernel Hilbert space with the kernel given by
 $$\sum_{u \in \mathcal{U}} \gamma_u \bigotimes_{j \in u}k,$$
 where $\mathcal{U} = \{u \subset \mathbb{N}: |u| < \infty\},$ and $(\gamma_u)_{u \in \mathcal{U}}$ is a sequence of non-negative numbers, into an appropriate $L^2$ space.
\end{abstract}

\section{Introduction and Problem Formulation}

Throughout the note we assume that all the Hilbert spaces we are considering are infinite-dimensional, since else the results are trivial. We also assume that they are separable.

Let $\Omega \neq \emptyset$ and $\mu$ be a $\sigma$-finite measure on $\Omega.$ We are considering a reproducing kernel $k: \Omega \times \Omega \rightarrow \mathbb{R}$ with the corresponding reproducing kernel Hilbert space $H = H(k).$ We assume that the identical embedding 
\begin{equation}\label{contEmb}
 S: H \rightarrow L^2(\Omega, \mu), \quad f \rightarrow f
\end{equation}
is continuous. This is exactly the case when every $f \in H$ is a representer of some equivalence class $\overline{f} \in L^2(\Omega, \mu).$
When no risk of confusion appears we denote $L^2 := L^2(\Omega, \mu).$

There is a certain interest in characterizing the situations in which $S$ is even compact. It stems from the fact that the problem of approximating $S$ (known as  $L^2$ function approximation) by deterministic algorithms in the information-based complexity framework is solvable in the worst-case setting exactly when $S$ is a compact operator. As recently shown in \cite{W22}, the same holds also for the randomized algorithms.   

It is well-known that 
\begin{equation}\label{kDiag}
\int_{\Omega} k(t,t) d\mu(t) < \infty 
\end{equation}
is equivalent to $S$ being a Hilbert-Schmidt operator. Furthermore, there is an ample supply of reproducing kernel Hilbert spaces and $L^2$ spaces for which the identical embedding is compact, but has an infinite trace, i.e. the embedding is compact, even though the condition (\ref{kDiag}) is violated. For details on those and related issues  we refer to \cite[Chapter 26]{NW12}.  
In Theorem \ref{MainTheorem} we give an if-and-only-if characterisation of the compactness of $S$ in terms of the reproducing kernel $k$.

In the recent years there was a surge of interest in investigating numerics of functions depending on inifinitely many variables, see e.g.\cite{Was12, Was12a, GHHR17} and the references therein. One of the most prominent problems is a special case of $L^2$ function approximation. The typical starting point is a reproducing kernel Hilbert space $H_{\gamma}$ and a space $L^2(\mathcal{X}, \mu^{\mathbb{N}})$, such that 
$$S_{\gamma}: H_{\gamma} \rightarrow L^2(\mathcal{X}, \mu^{\mathbb{N}}), \quad f \mapsto f,$$
is continuous.
Here $\gamma$ is a sequence of positive numbers, the so called weights, moderating the importance of different coordinates, and $\mathcal{X} \subset \Omega^{\mathbb{N}}$ is an appropriate subset of the sequence space. 
The space $H_{\gamma}$ is in a certain way build up of reproducing kernel Hilbert spaces $H$ of ``univariate" functions $f: \Omega \rightarrow \mathbb{R}.$ Assuming that the ``univariate" identical embedding $S:H \rightarrow L^2(\Omega, \mu)$ is compact we show in Theorem \ref{InfDimThm} that $S_{\gamma}$ is compact exactly when some sequence dependingonly on $\gamma$ and $\lVert S \rVert$ converges to $0.$ The setting of infinite-dimensional approximation is described in more detail in Section \ref{InfDim}.

\subsection{Embeddings of Reproducing Kernel Hilbert Spaces into $L^2$}

For the convenience of the reader we recall a few useful characterisations of compact operators.
\begin{proposition}\label{compactCharacterization}
  Let $X,Y,$ be Hilbert spaces and let $C:X \rightarrow Y$ be a linear operator. The following are equivalent: 
\begin{enumerate}
 \item $C$ is compact, i.e. the image of the unit ball of $X$ is precompact in $Y,$
 \item For every bounded sequence $(x_n)_{n}$ the sequence $(Cx_n)_n$ admits a convergent subsequence,
 \item For each sequence $(x_n)_n$ which converges weakly to $0$ the sequence $(Cx_n)_n$ converges strongly to $0,$
 \item $C^*$ is compact ( this is a special case of Schauder's Theorem).
\end{enumerate}
\end{proposition}

We define now integral operators
\begin{equation}\label{integralOperator}
 T_{H,H}: H \rightarrow H, \quad   T_{H,L^2}: H \rightarrow L^2(\Omega, \mu), \quad T_{L^2,L^2}: L^2(\Omega, \mu) \rightarrow L^2(\Omega, \mu),
\end{equation}
all given by the same formula
$$f \mapsto \left( s \mapsto \int_{\Omega}  k(s,t)f(t) d\mu(t) \right).$$

Our main result is the following.
\begin{theorem}\label{MainTheorem}
Let $\Omega \neq \emptyset$ and 
$k: \Omega \times \Omega \rightarrow \mathbb{R}$ be a reproducing kernel on $\Omega,$ with the corresponding reproducing kernel Hilbert space $H.$ Moreover, let $\mu$ be a $\sigma$-finite measure on $\Omega,$ for which $H \subset L^2(\Omega, \mu).$ Denote by  
$$ S: H \rightarrow L^2(\Omega, \mu), \quad f \mapsto f$$
the identical embedding. Then the following conditions are equivalent:
\begin{enumerate}
 \item $S$ is compact,
 \item Each (or equivalently: one) of the operators 
 $T_{H,H}, T_{H,L^2}, T_{L^2, L^2}$
 defined in (\ref{integralOperator})
 is compact,
 \item For each bounded (in the norm) sequence $(f_n)_{n \in \mathbb{N}} \subset H$ converging pointwise to $0$ one has $$\lim_{n \rightarrow \infty}\int_{\Omega}\left( \int_{\Omega} k(s,t)f_n(t) d\mu(t)\right)^2 d\mu(s) = 0.$$
\end{enumerate}
\end{theorem}

\begin{proof}[Proof of Theorem \ref{MainTheorem}]
  First we prove the equivalence $(1) \Leftrightarrow (2).$
  Note that we have
  $$T_{H, H} = S^*S, \quad T_{H, L^2} = SS^*S, \quad T_{L^2,L^2} = SS^*.$$
  Now the equivalence of $(1)$ and $(2)$ follows from the general theory of operators on Hilbert spaces, however, for completness we present short arguments. $(1) \Rightarrow (2)$ follows immediately from the ideal property of compact operators. 
  Compactness of $T_{H,H}$ implies the compactness of $S,$ because the singular values of $S$ are just the square roots of eigenvalues of $T_{H,H},$ i.e. either both sequences converge to $0,$ or they both do not. 
  
  By Schauder's Theorem and what we have shown so far, compactness of $T_{L^2,L^2} = SS^* = (S^*)^* S^*$ is equivalent to the compactness of $S^*,$ and thus also to the compactness of $S.$ 
  
  Suppose now that $SS^*S$ is compact.
Let $(f_n)_{n \in \mathbb{N}}$ be any sequence from the unit ball of $H.$ There exists a weakly convergent subsequence $(f_{n_k})_{k \in \mathbb{N}}.$ Denote its weak limit by $f.$ Using the Cauchy-Schwarz inequality we obtain
 $$\lVert S^*Sf_{n_k} - S^*Sf \rVert_H^2 = \langle SS^*S(f_{n_k} - f), S(f_{n_k} - f)  \rangle_{L^2} \leq 2 \lVert SS^*S(f_{n_k} - f) \rVert_{L^2} \lVert S \rVert.$$
 Now $SS^*S$ is a compact operator between Hilbert spaces, so in particular it maps sequences which converge weakly to $0$ to sequences which converge strongly to $0.$ Thus
 $$\lim_{k \rightarrow \infty} \lVert S^{*}Sf_{n_k} - S^{*}Sf \rVert_H = 0,$$
 and the compactness of $S^*S$ follows. This in turn implies the compactness of $S.$

 Now we prove $(2) \Leftrightarrow (3).$
 We claim that $(f_n)_n \subset H$ converges weakly to $0$ exactly when it is bounded and it converges pointwise to $0.$ As $\spn\{k(s, \cdot): s \in \Omega\}$ is dense in $H,$ it follows that $\spn\{\delta_s: s \in \Omega\},$
 where $\delta_s$ denotes the evaluation functional at the point $s,$ is dense in $H'.$ Let $Q \in \spn\{\delta_s: s \in \Omega\}.$ If $\delta_s(f_n) \rightarrow 0$ for each $s \in \Omega,$ then also $Q(f_n) \rightarrow 0.$ 
 Let now $\varphi \in H'$ be arbitrary. Given an $\epsilon > 0$ we may find a $Q \in \spn\{\delta_s: s \in \Omega\}$ with $\lVert \varphi - Q \rVert < \epsilon,$ and so
 $$\limsup_{n \rightarrow \infty}|\varphi(f_n)| \leq \limsup_{n \rightarrow \infty}\left( |\varphi(f_n) - Q(f_n)| + |Q(f_n)| \right) \leq \sup_n \lVert \varphi - Q \rVert \lVert f_n \rVert,$$
 i.e. $(f_n)_n$ converges weakly to $0.$

 The other implication is obvious (recall that by the general theory each weakly convergent sequence is bounded).
 
 Let now $f \in H.$ The equivalence $(2) \Leftrightarrow (3)$ follows from
 \begin{align*}
  & \lVert T_{H,L^2}f\rVert^2_{L^2} = 
 \int_{\Omega} \left( \int_{\Omega} k(s,t)  f(t)  d\mu(t)\right)^2d\mu(s)
 \end{align*}

 \begin{corollary}
  With the notation as in Theorem \ref{MainTheorem}:
  If $k \in L^2(\Omega \times \Omega, \mu \otimes \mu )$ then the embedding $S$ is compact.
 \end{corollary}
\begin{proof}
This follows from the well-known fact that $k \in L^2(\Omega \times \Omega, \mu \otimes \mu)$ implies the compactness of $T_{L^2,L^2},$ and from the equivalence of 1. and 2. in Theorem \ref{MainTheorem}.
\end{proof}

\begin{example}
 We show that if $k \notin L^2(\Omega \times \Omega, \mu \otimes \mu)$ then a priori we cannot say anything about the compactness of the embedding $S.$ To this end consider two finite measures $\nu = (\nu_i)_{i \in \mathbb{N}}$ and $\mu = (\mu_i)_{i \in \mathbb{N}}$ on $\mathbb{N}$ assigning a positive value to each natural number. We let $H = \ell^2(\nu),$ i.e.
 $$S: \ell^2(\nu) \rightarrow \ell^2(\mu), \quad f \mapsto f.$$
 Note that the reproducing kernel of $H$ is given by
 $$k(i,j) = \frac{\delta_{i,j}}{\nu_i}. $$
 One can easily calculate that
 \begin{itemize}
  \item $k \in L^2(\mathbb{N} \times \mathbb{N}, \mu \otimes \mu)$ if and only if $\sum_{i = 1}^{\infty} \left(   \frac{\mu_i}{\nu_i}\right)^2 < \infty,$
  \item $S$ is bounded if and only if $\sup_{i \in \mathbb{N}} \frac{\mu_i}{\nu_i} < \infty,$
\item $S$ is compact if and only if $\lim_{n \rightarrow \infty} \frac{\mu_i}{\nu_i} = 0.$
\end{itemize}
An example when $S$ is continuous but not compact is given e.g. by putting $\nu = \mu.$ On the other hand, putting
$$\mu_i = \frac{1}{i^2}, \quad \nu_i = \frac{\log(i+1)}{i^2}, \quad i \in \mathbb{N},$$ we see that
$$\lim_{i \rightarrow \infty} \frac{\mu_i}{\nu_i} = \lim_{i \rightarrow \infty} \frac{1}{\log(i+1)} = 0,$$
i.e. $S$ is compact, even though
$$\sum_{i = 1}^{\infty} \left( \frac{\mu_i}{\nu_i} \right)^2 = \sum_{i = 1}^{\infty} \frac{1}{\log(i+1)^2} = \infty,$$
i.e. $k \notin L^2(\mathbb{N} \times \mathbb{N}, \mu \otimes \mu).$
\end{example}

\end{proof}

\section{Infinite-variate function approximation.}\label{InfDim}
We shortly describe the typical setting for infinite-variate function approximation.
Let $\Omega \neq \emptyset$ and let $\mu$ be a \emph{probability} measure on $\Omega.$ Let $k$ be a reproducing kernel such that the corresponding Hilbert space $H = H(k)$ is compactly embedded into $L^2(\Omega, \mu)$. 

We additionally assume that the only constant function in $H(k)$ is the zero function. 
We treat $H$ as a space of univariate functions. Based on it we build up a space of functions with infinitely many variables. To this end denote
$$\mathcal{U} = \{u \subset \mathbb{N} : |u| < \infty\}.$$
For $u \in \mathcal{U}$ we put
$$k_u: \Omega^{\mathbb{N}} \times \Omega^{\mathbb{N}} \rightarrow \mathbb{R}, \quad (x,y) \mapsto \prod_{j \in u} k(x_j,y_j).$$ Denote
$$H_u = H(k_u).$$
For a family of non-negative numbers $\gamma = (\gamma_u)_{u \in \mathcal{U}},$ called weights, we set
$$\mathcal{U}_{\gamma} = \{u \in \mathcal{U} : \gamma_u > 0\}.$$
Put
$$\mathcal{X} = \{ x \in \Omega^{\mathbb{N}} : \sum_{u \in \mathcal{U}_{\gamma}} \gamma_u k_u(x,x) < \infty  \}.$$
We define a reproducing kernel $K_{\gamma}$ on $\mathcal{X}$ via
\begin{equation}\label{InfDimKernel}
K_{\gamma}: \mathcal{X} \times \mathcal{X} \rightarrow \mathbb{R}, \quad (x,y) \mapsto \sum_{u \in \mathcal{U}_{\gamma}} \gamma_u k_{u}(x,y).
\end{equation}
The corresponding reproducing kernel Hilbert space is denoted by $H_{\gamma}.$ We have
$$H_{\gamma} = \bigoplus_{u \in \mathcal{U}_{\gamma}} H_u.$$
This means that each function $f \in H_{\gamma}$ admits a unique decomposition
$$f = \sum_{u \in \mathcal{U}_{\gamma}} f_u, \quad f_u \in H_u.$$
The norm of $f$ is then given by
$$\lVert f \rVert_{H_{\gamma}}^2 = \sum_{u \in \mathcal{U}_{\gamma}} \frac{1}{\gamma_u} \lVert f_u \rVert_{H_u}^2.$$
We denote the univariate embedding
$$S: H \rightarrow L^2(\Omega, \mu).$$
We are interested in the compactness of the embedding
\begin{equation}\label{InfDimEmb}
S_{\gamma}: H_{\gamma} \rightarrow L^2\left(\mathcal{X}, \mu^{\mathbb{N}}\right). 
\end{equation}

\begin{theorem}\label{InfDimThm}
 Let  the univariate embedding $S$ be compact. Then the following conditions are equivalent:
 \begin{enumerate}
  \item $S_{\gamma}$ is compact,
  \item $\limsup_{j \rightarrow \infty} \gamma_{u_j} \lVert S \rVert^{2|u_j|} = 0,$ where $(u_j)_{j \in \mathbb{N}}$ is some enumeration of the elements of $\mathcal{U}_{\gamma}.$
 \end{enumerate}
\end{theorem}
\begin{proof}
 $(1) \Rightarrow (2).$ Suppose that $\limsup_{j \rightarrow \infty} \gamma_{u_j} C^{2|u_j|} = \epsilon > 0.$ In this case we will expose a sequence $(e_j)_{j \in \mathbb{N}}$ in $H_{\gamma}$ converging weakly to $0,$ for which $(S_{\gamma}e_j)_{j \in \mathbb{N}}$ does not converge strongly to $0.$ This will mean that $S_{\gamma}$ cannot be compact. To this end for $j \in \mathbb{N}$ let $e_j$ be any vector from $H_{u_j}$ satisfying
 $$\lVert e_j \rVert_{H_{\gamma}} = 1, \quad \lVert e_j \rVert_{L^2}^2 > \frac{1}{2}\gamma_{u_j} \lVert S \rVert^{2|u_j|}.$$
 Such a vector exists, since $\sqrt{\gamma_{u_j}}\lVert S \rVert^{|u_j|}$ is the norm of the identical embedding of $H_{u_j}$ into $L^2(\Omega^{|u_j|}, \mu^{|u_j|}),$ if we equip $H_{u_j}$ with the norm induced by $H_{\gamma}.$
As $H_{\gamma}$ is an orthogonal sum of $H_{u_j}, j \in \mathbb{N},$ the sequence $(e_j)_j$ forms an orthonormal system in $H_{\gamma}$. Thus for any $f \in H_{\gamma}$
$$ \sum_{j = 1}^{\infty} |\langle f , e_j \rangle_{H_{\gamma}}|^2  < \infty,$$
and so 
$$\lim_{j \rightarrow \infty} |\langle f , e_j \rangle_{H_{\gamma}}| = 0.$$
It follows that $(e_j)_j$ indeed converges weakly to $0.$ On the other hand there is a subsequence $(e_{j_k})_k$ such that the $L^2$-norms of all the $e_{j_k}$ are bounded away from $0$ by $\frac{1}{2}\epsilon,$ so $(e_j)_j$ cannot converge to $0$ in $L^2.$ 

$(2) \Rightarrow (1).$ Suppose that $(2)$ holds and let $(f_j)_{j \in \mathbb{N}}$ be any sequence in the unit ball of $H_{\gamma}.$ We will show that $(S_{\gamma}f_j)_{j \in \mathbb{N}}$ admits a convergent subsequence.
First of all note that $(2)$ actualy implies $\lim_{n \rightarrow \infty}\gamma_{u_j} \lVert S \rVert^{2|u_j|} = 0.$
Denote
$$H_n := \bigoplus_{j \leq n} H_{u_j},$$
and equip $H_n$ with the norm induced from $H_{\gamma}.$ Write
$$P_n: H_{\gamma} \rightarrow H_{\gamma}$$
for the orthogonal projection onto $H_n.$ Note that from $\sup_{j} \lVert f_j \rVert \leq 1$ we obtain
$$\sup_{j} \lVert P_n f_j \rVert \leq 1.$$
The univariate embedding $S$ was compact. Moreover, finite sums and tensor products of compact operators are also compact, and so we obtain that for each $n$ the sequence $(S_{\gamma}P_nf_j)_{j}$ admits a convergent subsequence. Use now the diagonal method to conclude that there is a subsequence $(f_{j_k})_{k \in \mathbb{N}}$ such that
$(S_{\gamma}P_nf_{j_k})_k$ converges for all $n.$ We claim that $(S_{\gamma}f_{j_k})_k$ is also convergent.
Put
$$g_n : = \lim_{k \rightarrow \infty} S_{\gamma}P_nf_{j_k}, \quad n \in \mathbb{N}.$$
We claim that 
\begin{enumerate}
 \item The sequence $(g_n)_n$ converges to some point $g$ in $L^2(\mathcal{X}, \mu^{\mathbb{N}}),$
 \item $\lim_{k \rightarrow \infty} S_{\gamma}f_{j_k} = g.$
 \end{enumerate}
We start by proving the first statement. Choose an $\epsilon > 0.$ For $n,m,k$ (for notational reasons we assume $n > m$) large enough we have
$$\lVert g_n - g_m \rVert_{L^2} \leq \lVert S_{\gamma}P_n f_{j_k} - S_{\gamma}P_m f_{j_k} \rVert_{L^2} + \frac{\epsilon}{2} = \lVert S_{\gamma}(P_n - P_m)f_{j_k} \rVert + \frac{\epsilon}{2} \leq \sqrt{\gamma_{u_{m+1}}}\lVert S \rVert^{|u_{m+1}|} + \frac{\epsilon}{2}.$$
 Thus $(g_n)_n$ is a Cauchy sequence, and the first statement follows.
 
 Now to the second statement. Choose an $\epsilon > 0.$ We can write
 $$\lVert S_{\gamma}f_{j_k} - g \rVert_{L^2} \leq \lVert S_{\gamma}f_{j_k} - S_{\gamma}P_n f_{j_k} \rVert_{L^2} + \lVert S_{\gamma}P_n f_{j_k} - g_n \rVert_{L^2} +\lVert g_n - g  \rVert_{L^2}.$$
 Take an $n$ large enough so that $\lVert g_n - g \rVert_{L^2} < \frac{\epsilon}{3}$ and $\sqrt{\gamma_{u_{n+1}}}\lVert S \rVert^{|u_{n+1}|} < \frac{\epsilon}{3}.$ This makes the first and the third summand small. Given this $n$ we may now chosse a $K$ satisfying $\lVert S_{\gamma}P_n f_{j_k} - g_n \rVert_{L^2} < \frac{\epsilon}{3}$ for all $k \geq K.$ Thus we have shown
 $$\lVert  S_{\gamma}f_{j_k} - g \rVert_{L^2} < \epsilon$$
 for all $k$ large enough. The second statement follows.
\end{proof}

\subsection*{Acknowledgements}
The author would like to thank Michael Gnewuch for helpful comments.
\bibliographystyle{siam}
\bibliography{References}

\begin{thebibliography}{1}

\bibitem{GHHR17}
{\sc M.~Gnewuch, M.~Hefter, A.~Hinrichs, and K.~Ritter}, {\em Embeddings of
  weighted {H}ilbert spaces and applications to multivariate and
  infinite-dimensional integration}, Journal of Approximation Theory, 222
  (2017), pp.~8--39.

\bibitem{NW12}
{\sc E.~Novak and H.~Wo\'zniakowski}, {\em Tractability of {M}ultivariate
  {P}roblems. Vol. 3: {S}tandard {I}nformation for {O}perators}, EMS Tracts in
  Mathematics, European Mathematical Society (EMS), Z\"urich, 2012.

\bibitem{Was12a}
{\sc G.~W. Wasilkowski}, {\em Liberating the dimension for function
  approximation and integration}, in Monte Carlo and Quasi-Monte Carlo Methods
  2012, L.~Plaskota and H.~Wo\'zniakowski, eds., Springer, Heidelberg, 2012,
  pp.~211--231.

\bibitem{Was12}
{\sc G.~W. Wasilkowski}, {\em Liberating the dimension for l2-approximation},
  J. Complexity, 28 (2012), pp.~304--319.

\bibitem{W22}
{\sc M.~Wnuk}, {\em Solvability by randomized algorithms}.
\newblock In preparation.

\end{thebibliography}
\end{document}